\documentclass[12pt]{amsart}

\usepackage{amsmath,amsfonts,amssymb,mathabx,latexsym,mathtools}
\usepackage{enumitem}
\usepackage[usenames, dvipsnames]{xcolor}
\usepackage{hyperref}
\usepackage{centernot}
\usepackage{tikz-cd}
\usepackage[all]{xy}
\usepackage[colorinlistoftodos]{todonotes}
\usepackage{stmaryrd}
\usepackage{adjustbox}

\usepackage{tikz,xypic}
\usetikzlibrary{calc, shapes, backgrounds,arrows,positioning,plotmarks,cd}
\tikzset{
	partial ellipse/.style args={#1:#2:#3}{
		insert path={+ (#1:#3) arc (#1:#2:#3)}
	}
}
\tikzcdset{every label/.append style = {font = \large}}
\newtheorem{theorem}{Theorem}[section]
\newtheorem{lemma}[theorem]{Lemma}
\newtheorem{proposition}[theorem]{Proposition}
\newtheorem{corollary}[theorem]{Corollary}

\theoremstyle{definition}
\newtheorem{definition}[theorem]{Definition}
\newtheorem{examplex}[theorem]{Example}
\newtheorem{problem}[theorem]{Problem}
\newenvironment{example}
  {\pushQED{\qed}\examplex}
  {\popQED\endexamplex}

\theoremstyle{remark}

\numberwithin{equation}{section}

\newcommand{\cC}{\ensuremath{\mathcal{C}}}

\renewcommand{\cR}{\ensuremath{\mathcal{R}}}

\newcommand{\cT}{\ensuremath{\mathcal{T}}}

\newcommand{\AP}{\ensuremath{\mathcal{AP}}}
\newcommand{\DP}{\ensuremath{\mathcal{DP}}}
\newcommand{\MP}{\ensuremath{\mathcal{MP}}}
\newcommand{\LH}{\ensuremath{\mathcal{LH}}}
\newcommand{\Y}{\ensuremath{\mathcal{Y}}}

\newcommand{\Cyc}{\ensuremath{\mathrm{Cyc}}}

\newcommand{\fkF}{\ensuremath{\mathfrak{F}}}
\newcommand{\fkI}{\ensuremath{\mathfrak{I}}}
\newcommand{\fkS}{\ensuremath{\mathfrak{S}}}

\newcommand{\st}{\ensuremath{\mathrm{st}}}

\DeclareMathOperator{\Fix}{Fix}
\DeclareMathOperator{\fix}{fix}

\DeclareMathOperator{\LR}{LR}

\begin{document}

%%%%%%%%%%%%%%%%%%%%%%%%%%%%%%%%%%%%%%%%%%%%%%%%%%%%%%%%%%%%
%  TITLE PAGE information
%%%%%%%%%%%%%%%%%%%%%%%%%%%%%%%%%%%%%%%%%%%%%%%%%%%%%%%%%%%%

%     [Short Title]{Full Title}
\title{On pattern avoidance in matchings and involutions}  

\author[J. J. Fang]{Jonathan J. Fang}
\address[JF]{Department of Mathematics, Brandeis University, Waltham, MA 02453}
\email{jjfang@brandeis.edu}
%\thanks{}

\author[Z. Hamaker]{Zachary Hamaker}
\address[ZH]{Department of Mathematics, University of Florida, Gainesville, FL 32601}
\email{zhamaker@ufl.edu}
%\thanks{}

\author[J. M. Troyka]{Justin M. Troyka}
\address[JT]{Department of Mathematics and Computer Science, Davidson College, Davidson, NC 28035}
\email{jutroyka@davidson.edu}

%    General info
%\subjclass[2010]{Primary 05E05; Secondary 14M15}

\date{\today}

%\dedicatory{}


\begin{abstract}
We study the relationship between two notions of pattern avoidance for involutions in the symmetric group and their restriction to fixed-point-free involutions.
The first is classical, while the second appears in the geometry of certain spherical varieties and generalizes the notion of pattern avoidance for perfect matchings studied by Jel\'inek.
The first notion can always be expressed in terms of the second, and we give an effective algorithm to do so.
We also give partial results characterizing the families of involutions where the converse holds.
As a consequence, we prove two conjectures of McGovern characterizing (rational) smoothness of certain varieties.
We also give new enumerative results, and conclude by proposing several lines of inquiry that extend our current work.
\end{abstract}

\maketitle
%\tableofcontents

%%%%%%%%%%%%%%%%%%%%%%%%%%%%%%%%%%%%%%%%%%%%%%%%%%%%%%%%%%%%%%%%
%
\section{Introduction}
%
%%%%%%%%%%%%%%%%%%%%%%%%%%%%%%%%%%%%%%%%%%%%%%%%%%%%%%%%%%%%%%%%
\label{sec:introduction}

Let $\fkS_n$ be the symmetric group and $\mathfrak{S} = \bigsqcup_n \fkS_n$.
For each sequence $w = w_1 \dots w_n$ with distinct real values, the \textbf{standardization} $\st(w)$ is the unique permutation  with the same relative order.
For $\pi, \sigma \in \mathfrak{S}$, say $\pi$ \textbf{contains} $\sigma$ if $\pi$ has a subsequence whose standardization is $\sigma$; otherwise say $\pi$ \textbf{avoids} $\sigma$.
For $\Pi \subseteq \fkS$, let $\fkS(\Pi)$ be the subset of $\fkS$ avoiding every $\pi \in \Pi$ and $\fkS_n(\Pi) = \fkS(\Pi) \cap \fkS_n$.
Containment is a partial order on $\fkS$, and we call the order ideals \textbf{(permutation) classes}.
Clearly, $\fkS(\Pi)$ is a class.
Moreover, every permutation class $\cC$ is $\fkS(\Pi)$ for some (possibly infinite) antichain $\Pi \subseteq \fkS$, which we call the \textbf{basis} of $\cC$.
See~\cite{vatter2015permutation} for further details.

We are interested in pattern avoidance for involutions.
Let 
\[
\fkI_n = \{\pi \in \fkS_n: \pi = \pi^{-1}\} \quad \mbox{and} \quad \fkF_{2n} = \{\tau \in \fkI_{2n}: \tau(i) \neq i\quad  \forall i \in[2n]\}
\]
be the set of involutions in $\fkS_n$ and fixed-point-free involutions in $\fkS_{2n}$.
Additionally, let  $\fkI = \bigsqcup_n \fkI_n$  and $\fkF = \bigsqcup_n \fkF_{2n}$.
The poset of pattern containment restricts naturally to involutions.
Given $\Pi \subseteq \fkS$, we define $\fkI_\fkS(\Pi) = \fkI \cap \fkS(\Pi)$ and $\fkF_\fkS(\Pi) = \fkF \cap \fkS(\Pi)$, which are the set of involutions and fixed-point-free involutions that avoid $\Pi$, respectively.
These notions are well studied, beginning with~\cite{simion1985restricted}.
See~\cite{bloom2013pattern,bona2016pattern} for more recent contributions with references to previous work.

There is another notion of pattern containment for $\fkF$ we call \emph{$\fkF$-containment} that requires the underlying cycle structure also be respected.
Similarly, one can define \textbf{$\fkF$-avoidance}, which is typically referred to as pattern avoidance for (perfect) matchings~\cite{jelinek2007dyck}.
This notion generalizes work on non-crossing and non-nesting ordered set partitions~\cite{chen2007crossings}, and is used to classify properties of certain algebraic varieties~\cite{mcgovern2011closures2,hamaker2020fixed}. 
See the discussion after Definition~\ref{def:I-avoidance} for a precise definition.

In this paper we study $\fkF$-avoidance and a generalization for $\fkI$ we call \textbf{$\fkI$-avoidance}.
A variant of $\fkI$-avoidance appears in work of McGovern~\cite{mcgovern2011closures} as a tool for classifying when certain algebraic varieties are smooth or rationally smooth.
The first usage of $\fkI$-avoidance we are aware of is in recent work by the Hamaker, Marberg and Pawlowski~\cite{hamaker2018involution,hamaker2019schur}, where it is used to classify algebraic properties of certain polynomials associated to involutions.

\begin{definition}
	\label{def:I-avoidance}
	We define \textbf{$\fkI$-containment} as the transitive closure on $\fkI$ of the following three relations: $\rho$ is less than $\tau$ if it can be obtained by
\begin{enumerate}
	\item deleting a 2-cycle from $\tau$ and standardizing,
	\item deleting a fixed point from $\tau$ and standardizing;
	\item deleting one entry from the 2-cycle $(i,i+1)$ and standardizing.
\end{enumerate}
Relation (3) converts a 2-cycle to a fixed point, but note that it applies only to a 2-cycle whose entries are adjacent.
For a 2-cycle $(i,i+1)$ in $\tau$, note that an application of (1) can be expressed as an application of (3) followed by (2), hence (1) is only sometimes a cover relation.
For $\tau, \rho \in \fkI$, we say $\tau$ \textbf{$\fkI$-avoids} $\rho$ if it does not $\fkI$-contain it. Finally, for $\cT \subseteq \fkI$, let $\fkI(\cT)$ be the set of involutions that $\fkI$-avoid every $\tau \in \cT$.

\end{definition}

\begin{example}
	Let $\tau=391786452 = (13)(29)(47)(58)(6)$.
	Then $\tau$ $\fkI$-contains the involution $\rho=1432 = (1)(24)(3)$ since $\rho$ can be obtained from $\tau$ by removing the 2-cycles $(29)$ and $(58)$ and standardizing to yield $(12)(35)(4)$, and then contracting the now adjacent 2-cycle $(12)$ to a fixed point (1) and standardizing to yield $(1)(24)(3)$.
	
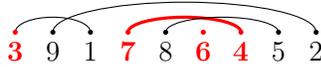
\[
	\begin{tikzpicture}[scale=.5]
		\draw (1,0) [partial ellipse=180:0:1 and .4];
	\draw[red, very thick] (4.5,0) [partial ellipse=180:0:1.5 and .4];
	\draw (5.5,0) [partial ellipse=180:0:1.5 and .4];
	\draw (4.5,0) [partial ellipse=180:0:3.5 and .8];
	\filldraw [red] (0,0) circle (1.5pt) node[anchor=north][scale=.9] {\textbf{3}};
	\filldraw [black] (1,0) circle (1.5pt) node[anchor=north][scale=.9] {9};
	\filldraw [black] (2,0) circle (1.5pt) node[anchor=north][scale=.9] {1};
	\filldraw [red] (3,0) circle (1.5pt) node[anchor=north][scale=.9] {\textbf{7}};
	\filldraw [black] (4,0) circle (1.5pt) node[anchor=north][scale=.9] {8};
	\filldraw [red] (5,0) circle (1.5pt) node[anchor=north][scale=.9] {\textbf{6}};
	\filldraw [red] (6,0) circle (1.5pt) node[anchor=north][scale=.9] {\textbf{4}};
	\filldraw [black] (7,0) circle (1.5pt) node[anchor=north][scale=.9] {5};
	\filldraw [black] (8,0) circle (1.5pt) node[anchor=north][scale=.9] {2};
	\end{tikzpicture}
	\]
However, $\tau$ $\fkI$-avoids the involution $2134 = (12)(3)(4)$ despite containing it in the ordinary sense.
\end{example}

Restricting $\fkI$-containment to $\fkF$, we obtain \textbf{$\fkF$-containment}.
Define \textbf{$\fkF$-avoidance} and $\fkF(\cR)$ analogously (for $\cR \subseteq \fkF$).
Here, only relation (1) is applicable since there are no fixed points.
By analogy with permutation classes, an \textbf{involution class} or \textbf{$\fkI$-class} is an order ideal in $\fkI$ under $\fkI$-containment.
Similarly, an \textbf{$\fkF$-class} is an order ideal in $\fkF$ under $\fkF$-containment.
The notion of $\fkI$-basis and $\fkF$-basis extend accordingly.
Our first collection of results describe $\fkI/\fkF$-bases for natural $\fkI/\fkF$-classes.

For $\Pi \subseteq \fkS$, it is relatively easy to see that $\fkI_\fkS(\Pi)$ is an $\fkI$-class and that $\fkF_\fkS(\Pi)$ is an $\fkF$-class.
Our first result, Theorem~\ref{t:basis}, shows their bases are finite when $\Pi$ is finite, with the size of basis elements increasing by at most a factor of $2$.
This is not a hard result, but it has significant implications.
If $n$ is the largest size of a permutation in $\Pi$, then checking up to size $2n$ is an effective algorithm for computing the $\fkI$-basis of $\fkI_\fkS(\Pi)$ and the $\fkF$-basis of $\fkF_\fkS(\Pi)$.
We used this algorithm to compute the bases in Table~\ref{tab:bases-algorithm}.
In Section~\ref{sec:mcgovern}, we use Theorem~\ref{t:basis} and a slight variant to prove some conjectures of McGovern characterizing his smoothness and rational smoothness results in terms of ordinary pattern avoidance.

\begin{table}
	\adjustbox{max width=\textwidth}{
		\begin{tabular}{c|c|c}
		basis & $\fkI$-basis & $\fkF$-basis\\
		\hline
		123 & 123, 14523, 34125, 351624, 456123 & 214365, 341265, 215634, 351624, 456123\\
		132 & 132, 35142, 465132  & 2143, 465132\\
		213  & 213, 42513, 546213 & 2143, 546213\\
		231 or 312 & 3412, 4231 & 3412, 632541\\
		321 & 321 & 4321

		\end{tabular}}
		\label{tab:bases-algorithm}
		\caption{$\fkF$-bases and $\fkI$-bases for classes avoiding one permutation in $S_3$.}
\end{table}

For $\cT \subseteq \fkI$ and $\cR \subseteq \fkF$, the sets $\fkI(\cT)$ and $\fkF(\cR)$ cannot necessarily be described in terms of ordinary pattern avoidance.
To see this, observe $\fkF_{2n}(2143) = \{\tau \in \fkF_{2n}: \tau(i) > n \ \forall i \in [n]\}$, which are sometimes called the \textbf{permutational matchings} since they are in easy bijection with $\fkS_n$.
The Marcus--Tardos theorem shows every permutation class $\cC$ except for $\fkS$ has an exponential growth rate~\cite{marcus2004excluded}, so since $|\fkF_{2n}(2143)| = n!$ we see $\fkF(2143) \neq \cC \cap \fkF$ for any permutation class $\cC$.
Since $\fkI_{2n}(12) = \fkF_{2n}(2143)$, the same applies for $\fkI$-avoidance.

Our second main result, Theorem~\ref{t:singleton}, shows that $2143$ and $21$ are in some sense the only obstructions to describing $\fkI/\fkF$-classes whose $\fkI/\fkF$-bases are singletons via ordinary pattern avoidance.
As a consequence, we derive Corollary~\ref{c:dichotomy}, which shows the growth rate of $\fkI$/$\fkF$-classes is either bounded above by an exponential function or below by $\lfloor n/2 \rfloor !$.
It is natural to ask a similar question for $\fkI/\fkF$-classes whose $\fkI/\fkF$-bases contain multiple elements.
\begin{problem}
\label{prob:pattern}
Characterize the sets $\cT \subseteq \fkI$ and $\cR \subseteq \fkF$ so that 
\[
\fkI(\cT) = \fkI_\fkS(\cT) \quad \mbox{and} \quad \fkF(\cR) = \fkF_\fkS(\cR).
\]

\end{problem}

We present some partial progress towards this problem in Theorem~\ref{t:bigger_bases}.

Since we are the first to systematically study $\fkI$-avoidance, basic enumerative questions remain open.
As a first step, we enumerate $\fkI(\tau)$ for $ \tau \in \fkI_3$.
\begin{theorem} 
	\label{t:I_3-avoidance}
	For $n \ge 1$,
\[
(a)\ |\fkI_n(321)| = \binom{n}{\lfloor n/2 \rfloor}, \quad (b)\ |\fkI_n(213)| = |\fkI_n(132)| = \sum_{k=0}^{\lfloor n/2 \rfloor} \binom{n-k}{k} k!,
\]
\[
(c)\ |\fkI_n(123)| = \sum_{k=1}^n {\textstyle \left\lfloor\frac{k}{2}\right\rfloor!\, \left\lfloor\frac{n-k}{2}\right\rfloor !} \,\binom{n - \left\lfloor \frac{k}{2} \right\rfloor - 1}{n-k}.
	\]

\end{theorem}
Theorem~\ref{t:singleton} implies that $\fkI(321) = \fkI_\fkS(321)$; the enumeration of $\fkI_\fkS(321)$ appears in~\cite{simion1985restricted}, and we give their formula in Theorem~\ref{t:I_3-avoidance}~(a).
The numbers in Theorem~\ref{t:I_3-avoidance}~(c) did not previously appear in OEIS, and they were added as~\cite[A339150]{sloane2018line} after the preprint of this paper was posted.
The numbers in Theorem~\ref{t:I_3-avoidance}~(b) are~\cite[A122852]{sloane2018line}.
Their ordinary generating function is a simple continued fraction~\cite{barry2009continued}.
More recently, Han showed~\cite{han2020hankel} this generating function is the $q=-1$ evaluation of a $q$-analogue of the Euler numbers.
Extending Han's work, Pan and Zeng gave the first organic combinatorial interpretation of these integers as certain labeled Motzkin paths called Andr\'e paths~\cite{pan2019branden}.
In Section~\ref{sec:enumeration} we prove Theorem~\ref{t:I_3-avoidance}, discuss the relationship between $\fkI_n(132)$ and prior work in greater detail, plus give a bijection from $\fkI_n(132)$ to Andr\'e paths.

For $\cR \subseteq \fkF$, it is easy to see the presence of fixed points in $\tau \in \fkI$ has no impact on whether or not $\tau \in \fkI(\cR)$.
This observation extends to a simple relationship of exponential generating functions.
Let $\cR \subseteq \fkF$ and let $F_\cR(x), I_\cR(x)$ be the exponential generating functions of $|\fkF_n(\cR)|$ and $|\fkI_n(\cR)|$, respectively.
Then
	\begin{equation}
	\label{eq:fixed-points}
	I_\cR(x) = e^x F_\cR(x).		
	\end{equation}
Here $|\fkF_n| = 0$ if $n$ is odd.
For $\fkF_\fkS(213) = \fkF(\{2143,546213\})$, a similar result appears in~\cite[Section 3.3]{bloom2013pattern}.
A version of Equation~\eqref{eq:fixed-points} upgraded to also account for the number of fixed points appears as Proposition~\ref{p:fixed-point}.
One immediate consequence is that the equality
\[
|\fkF_n(m{+}1\ \dots\ 2m\ 1\ \dots\ m)| =  |\fkF_n(2m\ 2m{-}1\ \dots\ 21)|
\]
due to Chen, Deng, Du, Stanley and Yan~\cite{chen2007crossings} extends to $\fkI$-avoidance.

\begin{corollary}
	\label{c:stanley}
	For $n,m$ positive integers,
	\[
|\fkI_n(m{+}1\ \dots\ 2m\ 1\ \dots\ m)| =  |\fkI_n(2m\ 2m{-}1\ \dots\ 21)|
\]
\end{corollary}

\noindent\textbf{Outline:} The remainder of the paper is structured as follows.
In Section~\ref{sec:proofs}, we prove our results on $\fkI/\fkF$-bases.
Using these results, we prove McGovern's conjectures in Section~\ref{sec:mcgovern}.
The proofs of Theorems~\ref{t:I_3-avoidance} and Equation~\eqref{eq:fixed-points}, as well as further consequences, appear in Section~\ref{sec:enumeration}.
We conclude in Section~\ref{sec:future} with a discussion of the geometric context for $\fkI$ and $\fkF$-avoidance, as well as some possible directions for future work.

\ 

\noindent\textbf{Acknowledgements:} The authors would like to thank Vincent Vatter for valuable comments and sharing his understanding of the history of $\fkF$-avoidance, Sergi Elizalde for pointing out related work, Qiongqiong Pan for helpful discussions related to $|\fkI(132)|$, Monty McGovern for sharing an unpublished preprint and slides related to his conjectures and the anonymous referees whose detailed feedback has significantly improved the quality of our exposition.
We would also like to thank the organizers of FPSAC 2018, where this work began.

\section{Results on Bases}
\label{sec:proofs}

For $I = \{i_1 < \dots < i_k\} \subseteq [n]$ and $\pi \in \fkS_n$, let $\st(\pi|_I)$ denote the standardization $\st(\pi_{i_1} \dots \pi_{i_k})$.
We state and prove our first main result, which shows for $\Pi \subseteq \fkS$ finite that the $\fkI$- and $\fkF$-bases of $\fkI_\fkS(\Pi) = \fkI \cap \fkS(\pi)$ and $\fkF_\fkS(\Pi) = \fkF \cap \fkS(\Pi)$ are finite.

\begin{theorem}
\label{t:basis}
	Let $\Pi$ be a set of permutations, each with size at most $n$.
	Then each element in the $\fkI$-basis of $\fkI_\fkS(\Pi)$ or $\fkF$-basis of $\fkF_\fkS(\Pi)$ has size at most $2n$.
\end{theorem}

\begin{proof}
	For $\Pi \subseteq \bigsqcup_{k=1}^n \fkS_n$, let $\tau \in \fkI$ have size $>2n$ and contain some $\pi \in \Pi$.
	We will show $\tau$ is not in the $\fkI$-basis of $\fkI_\fkS(\pi)$, hence every basis element has size at most $2n$.
	
	Since $\tau$ contains $\pi$, there exists $I = \{i_1 < i_2 < \dots < i_k\}$ so that $\st(\tau|_I) = \pi$.
	Let $J = \{\tau_i: i \in I\}$.
	By transpose symmetry, we see $\st(\tau|_J) = \pi^{-1}$.
	Let $\theta = \st(\tau|_{I \cup J})$.
	We claim $\theta$ has the following properties:
	\begin{enumerate}
		\item $\theta \notin \fkS(\Pi)$,
		\item $\theta$ has size at most $2k$,
		\item $\theta$ is $\fkI$-contained in $\tau$.
	\end{enumerate}
	Property (1) follows by construction since $\theta$ contains $\tau$, while (2) is immediate.
	Property (3) follows by observing that $\theta$ can be obtained from $\tau$ by deleting all fixed points and 2-cycles outside of $I \cup J$.
	Since $\tau$ $\fkI$-contains $\theta \in \bigsqcup_{k=1}^{2n} \fkI_k$ and $\theta$ contains $\pi$, the result follows for $\fkI$-avoidance.
	By assuming $\tau \in \fkF_\fkS(\Pi)$, the result extends immediately to $\fkF$-avoidance.
		\end{proof}

Define $\fkI'$-containment using only relations (1) and (2) from Definition~\ref{def:I-avoidance} and define $\fkI'$-avoidance analogously.
Since our proof of Theorem~\ref{t:basis} does not make use of relation (3) in the Definition~\ref{def:I-avoidance}, we have the following corollary.

\begin{corollary}
	\label{c:I'-avoidance}
	For $\Pi \subseteq \bigsqcup_{k=1}^n \fkS_n$, the $\fkI'$-basis of $\fkI_\fkS(\Pi)$ is contained in $\bigsqcup_{m=1}^{2n} \fkI_m$.
\end{corollary}

We will need Corollary~\ref{c:I'-avoidance} in Section~\ref{sec:mcgovern}.

Each $\pi \in \fkS_n$ has an associated plot in the $(x,y)$ plane, consisting of the points $(i,\pi(i))$ for all $i \in [n]$. The fixed points of $\pi$ are precisely the points of its plot that lie on the main diagonal $y=x$, and $\pi$ is an involution if and only if it is symmetric under reflection across the main diagonal $y=x$.

For permutations $\pi, \sigma$, the \textbf{skew sum} $\pi \ominus \sigma$ is the permutation obtained by juxtaposing the plots of $\pi$ and $\sigma$ so that $\sigma$ lies below and to the right of $\pi$.
That is, if $\pi \in \fkS_a$ and $\sigma \in \fkS_b$, then $\pi \ominus \sigma \in \fkS_{a+b}$ is defined by
\[ (\pi \ominus \sigma)(i) = \begin{cases}
b+\pi(i) & \text{if $1 \le i \le a$}; \\
\sigma(i-a) & \text{if $a+1 \le i \le a+b$}.
\end{cases} \]
We are now prepared to tackle our second main result.

\begin{theorem}
	\label{t:singleton}
	Let $\tau \in \fkI$, $\rho \in \fkF$.
	Then $\fkI(\tau) = \fkI_\fkS(\tau)$ (or $\fkF(\rho) = \fkF_\fkS(\rho)$) if and only if $\tau \in  \fkI(12)$ (or $\rho \in  \fkF(2143)$).
\end{theorem}

\begin{proof}
Assume that $\tau\in\fkI_n(12)$ (or $\rho\in\fkF_{2n}(2143)$).
There is a well-defined bijection
	\begin{align*}
	f:\fkI_n(12)&\rightarrow\fkS_{\lfloor \frac{n}{2}\rfloor}\\
	\tau&\mapsto\sigma \text{ iff } \sigma(k)={\textstyle \tau(\lceil \frac{n}{2}\rceil +k)} \text{ for } {\textstyle 1\leq k\leq \lfloor\frac{n}{2}\rfloor}.
	\end{align*}
Since $\fkF_{2n}(2143) = \fkI_{2n}(12)$, we see $f$ is well defined on $\fkF_{2n}(2143)$ as well.
Let $\sigma = f(\tau)$.
Then, if $\tau$ is of even size we may write $\tau = \upsilon \ominus \sigma$ for some $\upsilon\in S_{\frac{n}{2}}$.
Since $\tau$ is an involution $\upsilon = \sigma^{-1}$ so that $\tau = \sigma^{-1} \ominus \sigma$.
Similarly, if $\tau$ is of odd size we have $\tau = \sigma^{-1} \ominus 1 \ominus \sigma$.

\begin{figure}
\[ \begin{tikzpicture}[scale=0.5]
\draw (0,0) [very thick] rectangle (9,9);
\draw (0,0) [dashed] -- (9,9);
\draw (7,1) rectangle (4,5);
\draw (4,5) rectangle (2,8);
\node at (3,6.5) {\Large$\sigma$};
\node at (5.5,3) {\Large$\sigma^{-1}$};
\end{tikzpicture} \hspace{1in}
\begin{tikzpicture}[scale=0.5]
\draw (0,0) [very thick] rectangle (9,9);
\draw (0,0) [dashed] -- (9,9);
\draw (2,8) rectangle (5,4);
\draw (5,4) rectangle (7,1);
\node at (3.5,6) {\Large$\sigma$};
\node at (6,2.5) {\Large$\sigma^{-1}$};
\end{tikzpicture} \]
\caption{In the diagram on the left, all of $\sigma$ lies above the main diagonal of $\kappa$. In the diagram on the right, all of $\sigma^{-1}$ lies below the main diagonal. In these diagrams, we use different-sized rectangles to represent $\sigma$ and $\sigma^{-1}$, not squares and not of equal size, to emphasize the fact that $\sigma$ and $\sigma^{-1}$ may be spread out across various entries of $\kappa$, both vertically and horizontally. \label{fig:abovebelow}}
\end{figure}
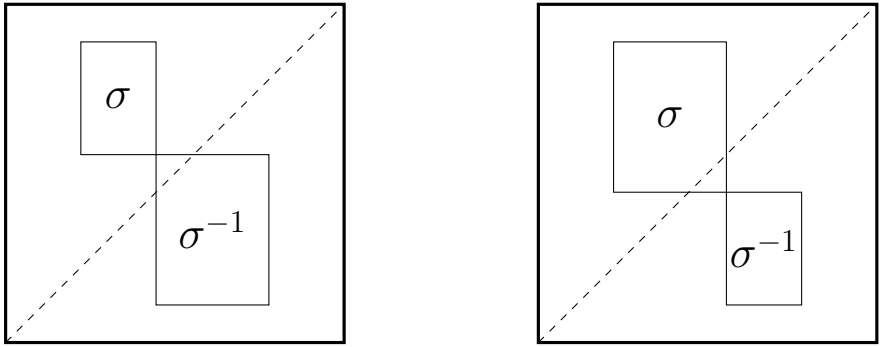

Let  $\tau = \sigma \ominus \sigma^{-1} \in \fkI(12)$ and $\kappa$ be an involution that contains $\tau$.
We will show $\kappa$ $\fkI$-contains $\tau$.
Since $\kappa$ is an involution, its plot is symmetric under reflection across the main diagonal.
In the copy of $\tau$ that occurs in $\kappa$, either all of $\sigma$ lies above the main diagonal of $\kappa$ or all of $\sigma^{-1}$ lies below the main diagonal of $\kappa$ --- see Figure \ref{fig:abovebelow}.
Without loss of generality, assume $\sigma$ lies above the main diagonal of $\kappa$.
In this case, the reflection of $\sigma$ across the main diagonal of $\kappa$ is another copy of $\sigma^{-1}$, and $\sigma$ and this new copy of $\sigma^{-1}$ are arranged in a skew sum --- so they form a new copy of $\tau$ in $\kappa$ (see Figure \ref{fig:newcopy}).
Furthermore, since the $\sigma$ and $\sigma^{-1}$ in this copy of $\tau$ are reflections of each other across the main diagonal of $\kappa$, the entries that share a $2$-cycle in $\tau$ also share a $2$-cycle in $\kappa$, and so this occurrence of $\tau$ is $\fkI$-contained in $\kappa$.

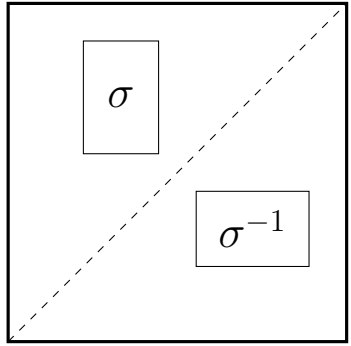
\begin{figure}
\[ \begin{tikzpicture}[scale=0.5]
\draw (0,0) [very thick] rectangle (9,9);
\draw (0,0) [dashed] -- (9,9);
\draw (8,2) rectangle (5,4);
\draw (4,5) rectangle (2,8);
\node at (3,6.5) {\Large$\sigma$};
\node at (6.5,3) {\Large$\sigma^{-1}$};
\end{tikzpicture} \]
\caption{When $\sigma$ lies entirely above the main diagonal of $\kappa$, the reflection of $\sigma$ across the main diagonal is a copy of $\sigma^{-1}$, and together these occurrences of $\sigma$ and $\sigma^{-1}$ form an occurrence of $\tau$ that is $\fkI$-contained in $\kappa$. \label{fig:newcopy}}
\end{figure}

The same argument extends immediately to $\kappa \in \fkF(2143)$.

Now let $\tau = \sigma \ominus 1 \ominus \sigma^{-1} \in \fkI(12)$ and $\kappa$ be an involution that contains $\tau$.
We can apply the same reasoning as in the case above, though care must be taken with the one fixed point of $\tau$, which is the ``$1$'' summand in the center of $\tau$.
In a given copy of $\tau$ in $\kappa$, if the fixed point of $\tau$ lies on the main diagonal of $\kappa$, then it is also a fixed point of $\kappa$, and so there is an $\fkI$-occurrence of $\tau$ consisting of the fixed point, the copy of $\sigma$, and the reflection of $\sigma$ just as before.
On the other hand, if the fixed point of $\tau$ does not lie on the main diagonal of $\kappa$, assume without loss of generality that it lies above the main diagonal.
Then there is an occurrence of $\sigma \ominus 21 \ominus \sigma^{-1}$ consisting of the fixed point, its reflection across the main diagonal of $\kappa$, the copy of $\sigma$, and the $\sigma^{-1}$ obtained by reflecting the copy of $\sigma$ across the main diagonal.
The center $21$ of the $\sigma \ominus 21 \ominus \sigma^{-1}$ is a $2$-cycle in $\kappa$, so we can delete the 1- and 2-cycles not in $\sigma \ominus 21 \ominus \sigma^{-1}$ and then contract the $21$ to form an $\fkI$-occurrence of $\tau$ in $\kappa$.

We have shown the notions of containment and $\fkI$-containment  coincide for $\tau\in\fkI(12)$, and likewise for $\fkF$-containment when $\rho\in\fkF(2143)$.
That is,
\[
	\fkI(\tau) = \fkI_\fkS(\tau) \quad \mbox{and} \quad \fkF(\rho) = \fkF_\fkS(\rho).		
\]
We prove the converse by contrapositive.
If $\tau$ $\fkI$-contains 12, then $\fkI(12) \subseteq \fkI(\tau)$, so 
\[
|\fkI_n(\tau)|\geq |\fkI_n(12)|=\left\lfloor \frac{n}{2}\right\rfloor !
\]
	for all $n$.
By the Marcus--Tardos theorem, we know that $|\fkS_n(\tau) \cap \fkI|$ is bounded by an exponential, so $\fkI(\tau) \neq \fkS(\tau) \cap \fkI = \fkI_\fkS(\tau)$.
The same argument applies when $\rho$ $\fkF$-contains $2143$.
\end{proof}

As a consequence, we see that the growth of $\fkI/\fkF$ classes exhibits a stark dichotomy.
\begin{corollary}
	\label{c:dichotomy}
	Let $\cT \in \fkI$, $\cR \in \fkF$.
	The sequences $\{|\fkI_{2n}(\cT)|\}$ and $\{|\fkI_{2n+1}(\cT)|\}$ (resp.\ $\{|\fkF_{2n}(\cR)|\}$) are bounded above by an exponential function if $\cT \cap \fkI(12) \neq \varnothing$ (resp.\ $\cR \cap \fkF(2143) \neq \varnothing$) and bounded below by $n!$ if $\cT \cap \fkI(12) = \varnothing$ (resp.\ $\cR \cap \fkF(2143) = \varnothing$).
\end{corollary}

\begin{proof}
	First suppose there exists $\tau \in \cT \cap \fkI(12)$.
	Then $\fkI_n(\cT) \subseteq \fkI_n(\tau)$, so
	\[
	|\fkI_n(\cT)| \le |\fkI_n(\tau)| = |\fkI_{\fkS_n}(\tau)| \le |\fkS_n(\tau)| \leq c^n
	\]
	for some constant $c$.
	Here, the equality is Theorem~\ref{t:singleton} and the final inequality is the Marcus--Tardos theorem.
	
	Now suppose $\cT \cap \fkI(12) = \varnothing$; then  $\fkI(12) \subseteq \fkI(\cT)$, hence
	\[
	n! \leq |\fkI_{2n}(\cT)|, |\fkI_{2n+1}(\cT)|. 
	\]
	The same argument proves the statement for $\fkF$-avoidance, with $2143$ instead of $21$.
\end{proof}

Using Theorem~\ref{t:singleton} and Corollary~\ref{c:dichotomy}, we present some partial progress towards solving Problem~\ref{prob:pattern}.
\begin{theorem}
	\label{t:bigger_bases}
	Let $\cT \subseteq \fkI$ and $\cR \subseteq \fkF$ be non-empty.
	Then (A) implies (B), and conversely (B) implies the weaker condition (C):
	\begin{enumerate}
		\item[$(A)$] $\cT \subseteq \fkI(12)$ (respectively $\cR \subseteq \fkF(2143)$).
		\item[$(B)$] $\fkI(\cT) = \fkI_\fkS(\cT)$ (respectively $\fkF(\cR) = \fkF_\fkS(\cR)$).
		\item[$(C)$] $\cT \cap \fkI(12)$ (respectively $\cR \cap \fkF(2143)$) is non-empty.
	\end{enumerate}
\end{theorem}

\begin{proof}
	To see that $(A)$ implies $(B)$,  we compute
\[
	\fkI(\cT) = \bigcap_{\tau\in\cT} \fkI(\tau)\\
	= \bigcap_{\tau\in\cT} \fkI_\fkS(\tau)\\
	= \fkI_\fkS(\cT)
	\]
with the second equality by Theorem 1.2. 
The implication $(B) \Rightarrow (C)$ follows from Corollary~\ref{c:dichotomy}.
Identical reasoning implies the equivalent results for $\fkF$-avoidance.
\end{proof}

From the $123$ row in Table~\ref{tab:bases-algorithm}, we see the reverse implication $(B) \Rightarrow (A)$ cannot hold: indeed, if $\cT$ is the $\fkI$-basis of $\fkI_\fkS(123)$, then $\cT$ satisfies $(B)$ (as can be checked by computer using Theorem~\ref{t:basis}) but does not satisfy $(A)$; the same holds for the $\fkF$-basis of $\fkF_\fkS(123)$.
To see $(C) \Rightarrow (B)$ does not hold, let $\cR =  \{2143, 456123\}$ and note $65872143$ is in $\fkF(\cR)$ and $\fkI(\cR)$ but it contains $2143$ as an ordinary permutation pattern.

\section{McGovern's conjectures}
\label{sec:mcgovern}

Theorem~\ref{t:basis} and Corollary~\ref{c:I'-avoidance} can be used to reduce two recent conjectures of McGovern to finite computations.
We give a brief overview of McGovern's work, which is at the intersection of Lie theory, representation theory and algebraic geometry.
Recall the Type A flag variety $FL(n) = GL(n)/B$ where $GL(n)$ is the set of complex invertible $n \times n$ matrices and $B$ is its Borel subgroup of upper triangular matrices.
The left action of the orthogonal group $O(n)$ decomposes $FL(n)$ into finitely many orbits $\{Y_\tau\}_{ \tau \in \fkI_n}$.
Similarly, for $n$ even the left action of the symplectic group $Sp(n)$ decomposes $FL(n)$ into orbits $\{Z_\rho\}_{\rho \in \fkF_n}$.
Equivalently,
\[
GL(n) = \bigsqcup_{\tau \in \fkI_n} O(n) M^\tau B \quad \mbox{and} \quad GL(2n) = \bigsqcup_{\rho \in \fkF_{2n}} Sp(2n) M^\rho B
\]
where $M^\tau$ is the permutation matrix of $\tau$.
There has been a great deal of work by many mathematicians trying to understand geometric properties of these orbits and their closures.
McGovern has given the following characterizations of smoothness (the variety is a manifold) and rational smoothness (a technical condition that is roughly equivalent to the variety satisfying Poincar\'e duality) for these orbits.

\begin{theorem}[{\cite[Theorem 1]{mcgovern2011closures2}}]
	\label{t:mcgovern-F}
	For $\rho \in \fkF_{2n}$, the orbit $Z_\rho$ is smooth if and only if it is rationally smooth if and only if $\rho \in \fkF(\Pi')$ where
	\begin{align*}
\Pi' =  	\{ & 351624, 64827153, 57681324, 53281764, 43218765, 65872143, 
21654387, \\
	&21563487, 34127856, 43217856, 34128765, 36154287, 21754836, 63287154, \\
&54821763, 46513287, 21768435\}.
	\end{align*}

\end{theorem}

Let $\fkI(21*43)$ be the set of involutions $\tau$ so that for every pair of cycles $(a,b),(c,d) \in \Cyc(\tau)$ with $a < b < c< d$, the number of fixed points between $b$ and $c$ is odd, i.e., $|\Fix(\tau) \cap [b,c]| \in 2\mathbb{N}+1$.
For example $21354 \in \fkI(21*43)$, but $2143 \notin \fkI(21*43)$ since there are an even number of fixed points (zero) between $1$ and $4$.

\begin{theorem}[{\cite[Theorem 1]{mcgovern2020closures}}]
	\label{t:mcgovern-I}
	For $\tau \in \fkI_n$, the orbit $Y_\tau$ is rationally smooth if and only if $\tau \in \fkI'(\Pi) \cap \fkI(21*43)$ where
	\begin{align*}
		\Pi = \{ & 14325,21543,32154, 154326, 124356, 351624, 132546, 426153, 153624, 351426,
		\\& 1243576, 2135467,
2137654, 4321576, 5276143, 5472163, 1657324, 4651327, \\ &57681324, 65872143, 13247856,34125768,34127856,64827153\}.
	\end{align*}	
\end{theorem}

One direction of Theorem~\ref{t:mcgovern-I} first appeared as~\cite[Theorem 1]{mcgovern2011closures}.
A similar result also applies for smoothness.

\begin{theorem}[{\cite[Theorem 2]{mcgovern2020closures}}]
	\label{t:mcgovern-I2}
For $\tau \in \fkI_n$, the orbit $Y_\tau$ is smooth if and only if $\tau \in \fkI'(\Pi \cup \{2143,1324\})$ with $\Pi$ as in Theorem~\ref{t:mcgovern-I}.
\end{theorem}

We prove McGovern's $\fkI'$ and $\fkF$-avoidance characterizations in Theorem~\ref{t:mcgovern-F} and Theorem~\ref{t:mcgovern-I2} hold when using ordinary pattern avoidance.
\begin{corollary}
\label{c:mcgovern}
Let $\Pi$ be as in Theorem~\ref{t:mcgovern-I} and $\Pi'$ as in Theorem~\ref{t:mcgovern-F}.
Then 
\begin{enumerate}
	\item $\fkI'(\Pi \cup \{2143,1324\}) = \fkI(\Pi \cup \{2143,1324\}) = \fkI_\fkS(\Pi \cup \{2143,1324\})$.
	\item $\fkF(\Pi') =  \fkF_\fkS(\Pi')$.
\end{enumerate}
\end{corollary}

\begin{proof}
By Theorem~\ref{t:basis} and Corollary~\ref{c:I'-avoidance}, the results will follow by checking equality up to $\fkI_{16}$.
We have used a computer to do so.
\end{proof}

The results in Corollary~\ref{c:mcgovern} were presented as conjectures at~\cite{mcgovern2019sectional}.
An analogous conjecture for Theorem~\ref{t:mcgovern-I} appears as Conjecture 4 in~\cite{mcgovern2019closures}, but our methods do not apply since $\fkI(21*43)$ is not an $\fkI$-class.
However, it seems plausible that our proof strategy for Theorem~\ref{t:basis} can be adapted to the ambient set $\fkI(21*43)$.

\section{Enumerative properties of $\fkF$- and $\fkI$-avoidance}
\label{sec:enumeration}

\subsection{Enumeration of $\fkI(132)$ and $\fkI(213)$}
First, we show both sets have the same enumeration using the reverse-complement symmetry.
Let $w_0 = n \dots 1 \in \fkS_n$.
The \textbf{reverse-complement} of $\pi = \pi_1 \dots \pi_n \in \fkS_n$ is $w_0 \cdot \pi \cdot w_0 = n{+}1{-}\pi_n \dots n{+}1{-}\pi_1$.
Implicitly, we require that $w_0$ and $\pi$ are members of the same symmetric group.
Since conjugation preserves cycle type, we see reverse-complement restricts to maps $\fkI_n \to \fkI_n$ and $\fkF_{2n} \to \fkF_{2n}$.
For $\Pi \subseteq \fkS$, let $w_0 \Pi w_0 = \{w_0 \pi w_0 : \pi \in \Pi\}$.
The following lemma is obvious from the definitions.

\begin{lemma}
\label{l:reverse-complement}
For $\Pi \subseteq \fkI$ and $\Pi' \subseteq \fkF$, we have
\[
w_0\fkI(\Pi)w_0 = \fkI(w_0\Pi w_0) \quad \mbox{and} \quad w_0 \fkF(\Pi') w_0 = \fkF(w_0 \Pi' w_0).
\]
\end{lemma}

For $\pi \in \fkS$, let  $C_k(\pi)$ be the number of $k$-cycles in $\pi$.
The \textbf{left} and \textbf{right endpoints} of a $2$-cycle $(a,b)$ are $a$ and $b$, respectively.

We now state and prove a $t$-analogue of Theorem~\ref{t:I_3-avoidance}~(b).

\begin{proposition}
\label{p:132}
\[\sum_{\tau \in \fkI_n(213)}  t^{C_2(\tau)} =  \sum_{\tau \in \fkI_n(132)}  t^{C_2(\tau)}  = \sum_{k=0}^{\lfloor n/2 \rfloor} \binom{n-k}{k} k! \,t^k  .\]
\end{proposition}

\begin{proof}
	Since reverse-complement preserves cycle type, the first equality follows from Lemma~\ref{l:reverse-complement}.
	To prove the second equality, we characterize $\tau \in \fkI(132)$.
	Let $\rho \in \fkF$ be the standardization of 2-cycles in $\tau$.
	Since $2143$ $\fkI$-contains $132$, we see $\rho \in \fkF(2143)$.
	Moreover, if $(i,j)$ is a 2-cycle in $\tau$ then every fixed point of $\tau$ must be greater than $i$.
	However, $\tau$ can have fixed points at any positions after the last left endpoint of $\rho$.
	Let $C_2(\tau) = k$.
	To construct $\tau$, we must choose $\rho \in \fkF_{2k}(2143)$ and place the $n-2k$ fixed points among the $k$ right endpoints in $\rho$.
	There are $k!$ choices for $\rho$ and $\binom{n-k}{k}$ ways to place the fixed points.
	Summing over $k$, the result follows.
\end{proof}
Theorem~\ref{t:I_3-avoidance}~(b) follows by setting $t=1$.
The weighted ordinary generating function for $\fkI(132)$ appears in recent work by Pan and Zeng~\cite{pan2019branden}.
Let
\[
[k]_{p,q} = \frac{p^k - q^k}{p-q} = \sum_{i=0}^{k-1} p^i q^{k-1-i} \quad \mbox{and} \quad \binom{n}{k}_{p,q} = \frac{[n]_{p,q} \dots [n-k+1]_{p,q}}{[k]_{p,q} \dots [1]_{p,q}} 
\]
and define $D_{n}(p,q,t)$ via the generating function
\begin{align}\label{pq-shift}
\sum_{n=0}^\infty D_{n+1}(p,q,t)x^n
=\cfrac{1}{1-x-\cfrac{\binom{2}{2}_{p,q} t\,x^2}{1-[2]_{p,q} x-
\cfrac{\binom{3}{2}_{p,q}t\,x^2}{1-[3]_{p,q} x-\cfrac{\binom{4}{2}_{p,q}t\,x^2}{1-[4]_{p,q} x-\cfrac{\binom{5}{2}_{p,q} t\,x^2}{1-\cdots}}}}}\ .
\end{align}
The $t = 1, p=1$ case is studied by~\cite{han2020hankel}.
\begin{corollary}
	\[\sum_{\tau \in \fkI_n(213)}  t^{C_2(\tau)} =  \sum_{\tau \in \fkI_n(132)}  t^{C_2(\tau)} = D_{n}(1,-1,t).
	\]
\end{corollary}
\begin{proof}
	The result follows from Proposition~\ref{p:132} and~\cite[Theorem 6]{pan2019branden}.
\end{proof}

Since $D_5(p,-1,t) = 1 + (p^2 +2)t + (p^2 - p +2)t^2$ has a negative term, we do not expect a direct 
combinatorial interpretation of $D_n(p,-1,t)$.
Using generating functions, Han proved the recurrence~\cite[Equation~(7.7)]{han2020hankel}
\begin{equation}
\label{eq:recurrence}
2|\fkI_n(132)| = 3 |\fkI_{n-1}(132)| + (n-1)|\fkI_{n-2}(132)| - (n-1)|\fkI_{n-3}(132)|.
\end{equation}
It is an interesting open problem to give a combinatorial proof of Equation~\eqref{eq:recurrence}.

Pan and Zeng give another combinatorial interpretation of $D_n(1,-1,t)$ in terms of objects they call Andr\'e paths.
Recall, a \textbf{Motzkin path} of \textbf{length} $n$ is a function $M: [n] \to \{U,D,L\}$ where the \textbf{height} $h_i(M) = |M^{-1}(U) \cap [i]| - |M^{-1}(D) \cap [i]|$ of the $i$th step is always non-negative for all $i \in [n]$.
Equivalently, we can write $M$ as the word $M_1 \dots M_n$ where $M_i = M(i)$, and depict $M$ by drawing $U$ as an up step, $L$ as a level step and $D$ as a down step.
An \textbf{Andr\'e path} is a Motzkin path where all level steps have even height and each down step $M_i$ is labeled with an integer between $1$ and $\lceil h_i(M)/2 \rceil$.
Let $\AP_n$ be the set of Andr\'e paths of length $n$ (see Figure~\ref{f:andre}).

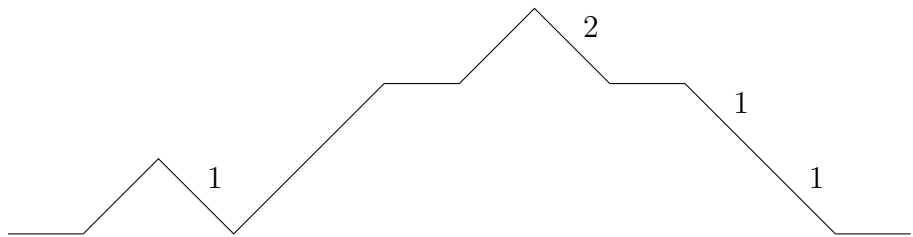
\begin{figure}
	\[\begin{tikzpicture}
	\draw (0,0) -- (1,0) -- (2,1) -- (3,0) -- (4,1) -- (5,2) -- (6,2) -- (7,3) -- (8,2) -- (9,2) -- (10,1) -- (11,0) -- (12,0);
	\node at (2.75,.75) {1};
	\node at (7.75,2.75) {2};
	\node at (10.75,.75) {1};
	\node at (9.75,1.75) {1};
\end{tikzpicture}\]
\caption{An Andr\'e path with underlying Motzkin path $LUDUULUDLDDL$. 
The only other Andr\'e path with the same Motzkin path has a 1 in place of the 2. \label{f:andre}}
\end{figure}

We exhibit an explicit bijection between Andr\'e paths and $\fkI(132)$.
An equivalent map can also be derived from the proof of~\cite[Theorem 6]{pan2020combinatorics}.
Recall a \textbf{Dyck path} is a Motzkin path with no $L$'s. Let $\DP_n$ represent the set of Dyck paths with half-length $n$ and $\MP_n$ represent the set of Motzkin paths with length $n$.
We start with a classical bijection due to Fran\c{c}on and Viennot~\cite[Theorem 5]{francon78histoires} between permutations and labeled \textbf{Laguerre histories}, which are Motzkin paths with two types of level steps denoted $L'$ and $L''$.
Let $\LH_n$ denote the set of Laguerre histories $H = H_1 \dots H_n$ with labeling $\lambda$ where
	$\lambda(H_j) \in  [h_j(H)+1]$.

\begin{lemma}[Fran\c{c}on--Viennot]
\label{lem:francon-viennot}
	There is a bijection $\chi: \fkS_{n+1} \rightarrow \LH_n$.
\end{lemma}
We give a brief description of the Fran\c con--Viennot bijection.
Recall a binary tree is a rooted tree in which each vertex has no children, a right child, a left child, or a left child and a right child.
We denote the four cases by $\varnothing, \{r\}, \{\ell\}$ and $\{\ell,r\}$, respectively.
Given a word with distinct real entries $w = w_1 \dots w_n$, let $i$ be the index so that $w_i$ is minimal in $w$ and define the tree $T_w$ recursively with root $i$ whose left and right childen are the roots of $T_{w_1 \dots w_{i-1}}$ and $T_{w_{i+1} \dots w_n}$, respectively.
This procedure produces a bijection from $\fkS_n$ to  $\cT_n$, the set of binary trees with vertex labels $1,2,\ldots n$ so that the label of each vertex is less than its children.

There is also a bijection from $\LH_{n-1}$ to $\cT_n$.
Given $(H = H_1 \dots H_{n-1},\lambda) \in \LH_{n-1}$, we construct a tree by inserting the vertices $1,2,\ldots n$ successively.
Begin with an empty tree at the 0th step that has one available position to fill.
For $i \in [n-1]$, the set of children for the $i$th vertex will be
\[
\begin{array}{ll}
\{\ell,r\} & \text{if $H_i = U$},\\
\{\ell\} & \text{if $H_i = L'$}, \\
\{r\} & \text{if $H_i = L''$}, \\
\varnothing & \text{if $H_i = D$}.
\end{array}
\]
At step $i$, the number of possible positions to insert the $i$th vertex is $1+h_i(H)$.
We insert the vertex $i$ at the $\lambda(H_i)$th position starting from the left. Finally we insert the $n$th vertex in the only available position remaining.
This map is invertible, and $\chi$ is the composition $\sigma \mapsto T \mapsto (H,\lambda)$ where $T$ is the intermediate binary tree.

Let $\DP_{n+1}'$ be the set of labeled Dyck paths $(M,\mu)$ with $\mu(M_i) = 1$ if $M_i = U$ and $1 \leq \mu(M_i) \leq \lceil h_i(D)/2 \rceil$ if $M_i = D$.
\begin{lemma}
\label{l:dyck-path}
	There is a bijection $\phi: \DP_{n+1}' \rightarrow \LH_n $.
\end{lemma}
\begin{proof}
	Let $(M = M_1\ldots M_{2n+2},\mu) \in \DP'_{n+1}$.
	To define $\phi((M,\mu)) = (H,\lambda)$, let
	\[
	H_i = \begin{cases}
	U & \text{if $M_{2i}M_{2i+1}=UU$};\\
	D & \text{if $M_{2i}M_{2i+1}=DD$};\\
	L' & \text{if $M_{2i}M_{2i+1}=UD$};\\
	L'' & \text{if $M_{2i}M_{2i+1}=DU$}.
	\end{cases}
	\]
	By construction, the height $h_{i}(H) = \frac{j}{2}$ where $j = h_{2i}(M)$.
	Each $U$ in $H$ corresponds to the next $D$ at the same level.
	For $H_i = U$, let $\ell(i)$ be the index so that $H_{\ell(i)} = D$ corresponds to $H_i$.
	We give the resulting path labels according to the weights of the original path:
	\[
	\lambda(L_i)= \begin{cases}
 \mu(M_{2\ell(i)+1}) & \text{if $H_i = U$};\\
 \mu(M_{2i}) & \text{if $H_i = D$};\\
 \mu(M_{2i+1}) & \text{if $H_i = L'$};\\
 \mu(M_{2i}) & \text{if $H_i = L''$}.	
 \end{cases}
	\]
Note this is a label preserving bijection, in the sense that the two labelings have identical outputs when $\mu$ is restricted to down steps in $M$.
\end{proof}

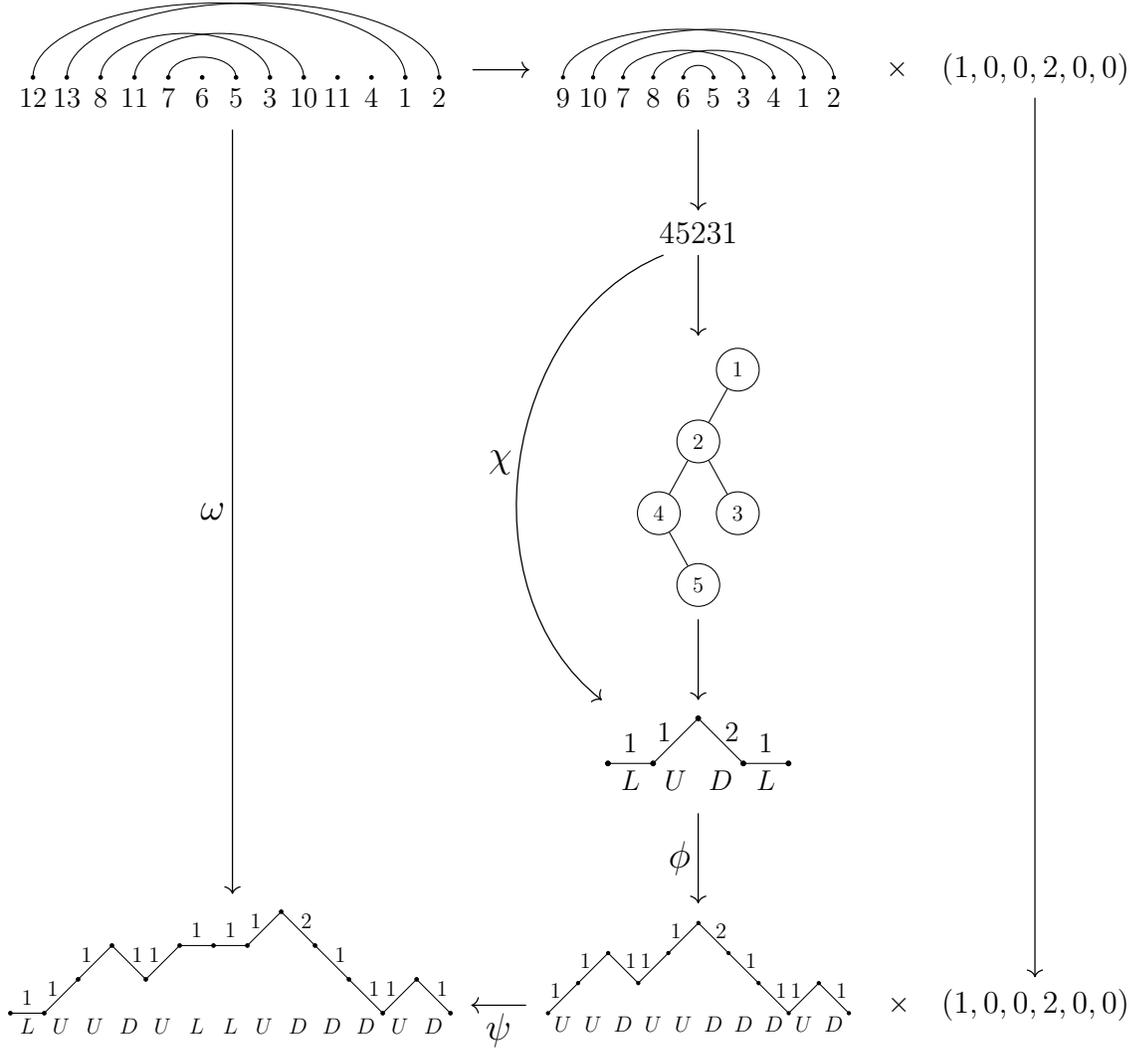
\begin{figure}
	\begin{tikzcd}[row sep=6ex, column sep=4ex]
	\begin{tikzpicture}[scale=.45]
	
	\draw (5.5,0) [partial ellipse=180:0:5.5 and 2.2];
	\draw (6.5,0) [partial ellipse=180:0:5.5 and 2.2];
	\draw (4.5,0) [partial ellipse=180:0:2.5 and 1.3];
	\draw (5.5,0) [partial ellipse=180:0:2.5 and 1.3];
	\draw (5,0) [partial ellipse=180:0:1 and .6];
	\filldraw [black] (0,0) circle (1.5pt) node[anchor=north][scale=.9] {12};
	\filldraw [black] (1,0) circle (1.5pt) node[anchor=north][scale=.9] {13};
	\filldraw [black] (2,0) circle (1.5pt) node[anchor=north][scale=.9] {8};
	\filldraw [black] (3,0) circle (1.5pt) node[anchor=north][scale=.9] {11};
	\filldraw [black] (4,0) circle (1.5pt) node[anchor=north][scale=.9] {7};
	\filldraw [black] (5,0) circle (1.5pt) node[anchor=north][scale=.9] {6};
	\filldraw [black] (6,0) circle (1.5pt) node[anchor=north][scale=.9] {5};
	\filldraw [black] (7,0) circle (1.5pt) node[anchor=north][scale=.9] {3};
	\filldraw [black] (8,0) circle (1.5pt) node[anchor=north][scale=.9] {10};
	\filldraw [black] (9,0) circle (1.5pt) node[anchor=north][scale=.9] {11};
	\filldraw [black] (10,0) circle (1.5pt) node[anchor=north][scale=.9] {4};
	\filldraw [black] (11,0) circle (1.5pt) node[anchor=north][scale=.9] {1};
	\filldraw [black] (12,0) circle (1.5pt) node[anchor=north][scale=.9] {2};
	
	\end{tikzpicture}\arrow[dddd, "\omega"{left}]\arrow[r]&\begin{tikzpicture}[scale=.4]
	
	\draw (4,0) [partial ellipse=180:0:4 and 1.6];
	\draw (5,0) [partial ellipse=180:0:4 and 1.6];
	\draw (4,0) [partial ellipse=180:0:2 and .9];
	\draw (5,0) [partial ellipse=180:0:2 and .9];
	\draw (4.5,0) [partial ellipse=180:0:.5 and .4];
	\filldraw [black] (0,0) circle (1.5pt) node[anchor=north][scale=.9] {9};
	\filldraw [black] (1,0) circle (1.5pt) node[anchor=north][scale=.9] {10};
	\filldraw [black] (2,0) circle (1.5pt) node[anchor=north][scale=.9] {7};
	\filldraw [black] (3,0) circle (1.5pt) node[anchor=north][scale=.9] {8};
	\filldraw [black] (4,0) circle (1.5pt) node[anchor=north][scale=.9] {6};
	\filldraw [black] (5,0) circle (1.5pt) node[anchor=north][scale=.9] {5};
	\filldraw [black] (6,0) circle (1.5pt) node[anchor=north][scale=.9] {3};
	\filldraw [black] (7,0) circle (1.5pt) node[anchor=north][scale=.9] {4};
	\filldraw [black] (8,0) circle (1.5pt) node[anchor=north][scale=.9] {1};
	\filldraw [black] (9,0) circle (1.5pt) node[anchor=north][scale=.9] {2};
	
	\end{tikzpicture}\arrow[r, phantom, "\times"]\arrow[d]&[.25ex](1,0,0,2,0,0)\arrow[dddd]\\
	&45231\arrow[dd, bend right=60, "\chi"{left}]\arrow[d]&\\
	&\begin{tikzpicture}[scale=.7]
	\node[circle,draw][scale=.7](z){1}
	child{
		node[circle,draw][scale=.7]{2} 
		child{
			node[circle,draw][scale=.7] {4}
			child[missing]{}
			child{ node[circle,draw][scale=.7] {5} }
		} 
		child{node[circle,draw][scale=.7] {3}} 
	}
	child[missing]{};
	\end{tikzpicture}\arrow[d]&\\
	&\begin{tikzpicture}[scale=.6]
	\draw (0,0) -- (1,0) -- (2,1) -- (3,0) -- (4,0);
	\filldraw [black] (0,0) circle (1.5pt);
	\filldraw [black] (1,0) circle (1.5pt);
	\filldraw [black] (2,1) circle (1.5pt);
	\filldraw [black] (3,0) circle (1.5pt);
	\filldraw [black] (4,0) circle (1.5pt);
	
	\node[scale=.9] at (.5,.25) {1};
	\node[scale=.9] at (1.25,.5) {1};
	\node[scale=.9] at (2.75,.5) {2};
	\node[scale=.9] at (3.5,.25) {1};
	
	\node[scale=.9] at (.5,-.6) {$L$};
	\node[scale=.9] at (1.5,-.6) {$U$};
	\node[scale=.9] at (2.5,-.6) {$D$};
	\node[scale=.9] at (3.5,-.6) {$L$};
	\end{tikzpicture}\arrow[d, "\phi"{left}]&\\
	\begin{tikzpicture}[scale=.45]
	\draw (-1,0) -- (0,0) -- (1,1) -- (2,2) -- (3,1) -- (4,2) -- (5,2) -- (6,2) -- (7,3) -- (8,2) -- (9,1) -- (10,0) -- (11,1) -- (12,0);
	
	\filldraw [black] (-1,0) circle (1.5pt);
	\filldraw [black] (0,0) circle (1.5pt);
	\filldraw [black] (1,1) circle (1.5pt);
	\filldraw [black] (2,2) circle (1.5pt);
	\filldraw [black] (3,1) circle (1.5pt);
	\filldraw [black] (4,2) circle (1.5pt);
	\filldraw [black] (5,2) circle (1.5pt);
	\filldraw [black] (6,2) circle (1.5pt);
	\filldraw [black] (7,3) circle (1.5pt);
	\filldraw [black] (8,2) circle (1.5pt);
	\filldraw [black] (9,1) circle (1.5pt);
	\filldraw [black] (10,0) circle (1.5pt);
	\filldraw [black] (11,1) circle (1.5pt);
	\filldraw [black] (12,0) circle (1.5pt);
	
	\node[scale=.7] at (-.5,.25) {1};
	\node[scale=.7] at (.25,.5) {1};
	\node[scale=.7] at (1.25,1.5) {1};
	\node[scale=.7] at (2.75,1.5) {1};
	\node[scale=.7] at (3.25,1.5) {1};
	\node[scale=.7] at (4.5,2.25) {1};
	\node[scale=.7] at (5.5,2.25) {1};
	\node[scale=.7] at (6.25,2.5) {1};
	\node[scale=.7] at (7.75,2.5) {2};
	\node[scale=.7] at (8.75,1.5) {1};
	\node[scale=.7] at (9.75,.5) {1};
	\node[scale=.7] at (10.25,.5) {1};
	\node[scale=.7] at (11.75,.5) {1};
	
	\node[scale=.7] at (-.5,-.6) {$L$};
	\node[scale=.7] at (.5,-.6) {$U$};
	\node[scale=.7] at (1.5,-.6) {$U$};
	\node[scale=.7] at (2.5,-.6) {$D$};
	\node[scale=.7] at (3.5,-.6) {$U$};
	\node[scale=.7] at (4.5,-.6) {$L$};
	\node[scale=.7] at (5.5,-.6) {$L$};
	\node[scale=.7] at (6.5,-.6) {$U$};
	\node[scale=.7] at (7.5,-.6) {$D$};
	\node[scale=.7] at (8.5,-.6) {$D$};
	\node[scale=.7] at (9.5,-.6) {$D$};
	\node[scale=.7] at (10.5,-.6) {$U$};
	\node[scale=.7] at (11.5,-.6) {$D$};
	\end{tikzpicture}&\begin{tikzpicture}[scale=.4]
	\draw (0,0) -- (1,1) -- (2,2) -- (3,1) -- (4,2) -- (5,3) -- (6,2) -- (7,1) -- (8,0) -- (9,1) -- (10,0);
	
	\filldraw [black] (0,0) circle (1.5pt);
	\filldraw [black] (1,1) circle (1.5pt);
	\filldraw [black] (2,2) circle (1.5pt);
	\filldraw [black] (3,1) circle (1.5pt);
	\filldraw [black] (4,2) circle (1.5pt);
	\filldraw [black] (5,3) circle (1.5pt);
	\filldraw [black] (6,2) circle (1.5pt);
	\filldraw [black] (7,1) circle (1.5pt);
	\filldraw [black] (8,0) circle (1.5pt);
	\filldraw [black] (9,1) circle (1.5pt);
	\filldraw [black] (10,0) circle (1.5pt);
	
	\node[scale=.7] at (.25,.5) {1};
	\node[scale=.7] at (1.25,1.5) {1};
	\node[scale=.7] at (2.75,1.5) {1};
	\node[scale=.7] at (3.25,1.5) {1};
	\node[scale=.7] at (4.25,2.5) {1};
	\node[scale=.7] at (5.75,2.5) {2};
	\node[scale=.7] at (6.75,1.5) {1};
	\node[scale=.7] at (7.75,.5) {1};
	\node[scale=.7] at (8.25,.5) {1};
	\node[scale=.7] at (9.75,.5) {1};
	
	\node[scale=.7] at (.5,-.6) {$U$};
	\node[scale=.7] at (1.5,-.6) {$U$};
	\node[scale=.7] at (2.5,-.6) {$D$};
	\node[scale=.7] at (3.5,-.6) {$U$};
	\node[scale=.7] at (4.5,-.6) {$U$};
	\node[scale=.7] at (5.5,-.6) {$D$};
	\node[scale=.7] at (6.5,-.6) {$D$};
	\node[scale=.7] at (7.5,-.6) {$D$};
	\node[scale=.7] at (8.5,-.6) {$U$};
	\node[scale=.7] at (9.5,-.6) {$D$};
	\end{tikzpicture}\arrow[l, "\psi"]\arrow[r, phantom, "\times"]&(1,0,0,2,0,0)
	\end{tikzcd}
	\caption{An example of the bijection between Andr\'{e} paths and $\fkI(132)$ using the involution 659421783}
	\label{b:andre}
\end{figure}

Let $\Y_{n,k}$ be the set of weak compositions of $n$ into $k$ parts.
Let $\AP_{n,j}$ be the set of Andr\'e paths with $j$ $L$'s.

\begin{lemma}[{\cite[Lemma 4]{pan2019branden}}]
	There is a bijection 
	$
	\psi:\AP_{n,n-2k}\rightarrow \Y_{n,k}\times \DP_k'.
	$
\end{lemma}
\begin{proof}
	When $k > \lfloor n/2 \rfloor$, both sets are empty.
	For $0 \leq k \leq \lfloor n/2 \rfloor$, we construct $\psi^{-1}$.
	Let $M \in \DP_k$, and observe $h_i(M)$ is even if and only if $i$ is even.
	Given $(y_1,\dots,y_k) \in \Y_{n,k}$, we extend $M$ to a Motzkin path by adding $y_i$ $L$'s after $M_{2i}$.
	This process is invertible.
	The paths in $\AP_{n,n-2k}$ and $\DP'_k$ then have the same number of $D$'s at the same heights, so the same label is applied to both.
\end{proof}

\begin{proposition}
	There is a bijection $\omega: \fkI_n(132) \rightarrow \AP_{n}$ so that $C_1(\tau)$ is the number of $L$'s in $\omega(\tau)$.
	\end{proposition}
\begin{proof}
	Let $\fkI_{n,k}(132) = \{\tau \in \fkI_n(132): C_2(\tau) = k\}$.
	We have the bijection
	\[
	\fkI_{n,k}(132) \to \fkF_{k}(2143) \times \Y_{n,k} \to \fkS_k \times \ \xrightarrow{(\phi \circ \chi,id)} \DP'_k \times Y_{n,k} \xrightarrow{\psi} \AP_{n,n-2k}.
	\]
	Here, the first map comes by recording the position of fixed points and the second by the bijection $\fkF_{2n}(2143) \to \fkS_n$.
	Summing over $k$, we obtain the desired bijection.
\end{proof}

\subsection{Enumeration of $\fkI(123)$}
In this section, we prove Theorem~\ref{t:I_3-avoidance}~(c):
\[ |\fkI_n(123)| = \sum_{k=1}^n {\textstyle \left\lfloor\frac{k}{2}\right\rfloor!\, \left\lfloor\frac{n-k}{2}\right\rfloor !} \,\binom{n - \left\lfloor \frac{k}{2} \right\rfloor - 1}{n-k}. 
\]
Assume $n \ge 1$ for the remainder of this section.

For $\tau \in \fkI$, let $\Cyc(\tau) = \{(a,b): a \leq \tau(a) = b\}$.
Given $(a,b),(c,d) \in \Cyc(\tau)$, say $(a,b)$ is \textbf{to the left} of $(c,d)$ (and $(c,d)$ is \textbf{to the right} of $(a,b)$) if $b < c$.
In this case, we say $(a,b)$ and $(c,d)$ are \textbf{independent}.
We say $(a,b) \in \Cyc(\tau)$ is a \textbf{left-to-right minimum} if there is no cycle to the left of $(a,b)$.
Let $\LR(\tau)$ be the set of all entries of left-to-right minima in $\tau$; that is, $\LR(\tau) = \bigcup \{a,b\}$ where the union ranges over left-to-right minima $(a,b)$ of $\tau$.
% Let $\text{LR}(\tau)$ be the set of left-to-right minima in $\tau$ and $I(\tau) = \bigcup_{(a,b) \in \text{LR}(\tau)} \{a,b\}$.
For example, with $\tau = 426153$ we have
\[
\Cyc(\tau) = \{(1,4), (2,2), (3,6), (5,5)\}
\]
with independent pairs $\{(1,4),(5,5)\}, \{(2,2),(3,6)\}$ and $\{(2,2),(5,5)\}$.
Then the left-to-right minima of $\tau$ are $(1,4)$ and $(2,2)$, and $\LR(\tau) = \{1,2,4\}$.

An $\fkI$-occurrence of $12$ in an involution $\tau$ corresponds to an independent pair of cycles of $\tau$.
Similarly, an $\fkI$-occurrence of $123$ in $\tau$ corresponds to three pairwise independent cycles of $\tau$.
Using this observation, we show an involution in $\fkI(123)$ can be constructed as the union of two involutions in $\fkI(12)$.
This is analogous to the classical characterization of $\fkS(123)$ as the set of permutations that are a union of two permutations in $\fkS(12)$, i.e.\ decreasing subsequences.
This viewpoint justifies our use of ``left-to-right minimum''.

\begin{proposition} \label{prop:union12}
Let $\tau \in \fkI$.
Then $\tau \in \fkI(123)$ if and only if $[n]$ can be partitioned into two sets, $[n] = I \sqcup J$, such that $\st(\tau|_I)$ and $\st(\tau|_J)$ are involutions that $\fkI$-avoid $12$.
Moreover, when this condition holds, we can take $I = \LR(\tau)$ and $J = [n] \smallsetminus \LR(\tau)$.\end{proposition}

\begin{proof}
Let $\tau \in \fkI(123)$. Define $I = \LR(\tau)$, $J = [n] \smallsetminus \LR(\tau)$, $\alpha = \st(\tau|_I)$, and $\beta = \st(\tau|_J)$.
We claim that $\alpha,\beta \in \fkI(12)$.
First, note $(a,b),(c,d) \in \Cyc(\alpha)$ cannot be independent, else one would not be a left-to-right minimum.
Therefore $\alpha \in \fkI(12)$.
Similarly, if $\beta \notin \fkI(12)$ then $\beta$ has a cycle $(a,b)$ to the left of some $(c,d)$.
Since $(a,b)$ is not a left-to-right minimum in $\tau$ (by definition of $J$), it has a cycle to its left in $\tau$, resulting in three pairwise independent cycles in $\tau$, a contradiction.

Now let $\tau \in \fkI$ and let $[n] = I \sqcup J$ be a partition such that $\alpha = \st(\tau|_I)$ and $\beta = \st(\tau|_J)$ are in $\fkI(12)$.
Suppose $\tau \notin \fkI(123)$, meaning that $\tau$ has three pairwise independent cycles. Two of these cycles must be in the same part of the partition (either $I$ or $J$), resulting in either $\alpha$ or $\beta$ having two independent cycles.
This contradicts the fact that $\alpha$ and $\beta$ are both in $\fkI(12)$.
\end{proof}

\begin{lemma} \label{lem:ltr-min1}
Let $\tau \in \fkI_n(123)$, and set $k = |\LR(\tau)|$.
Then $\{1, \ldots, \lfloor k/2 \rfloor + 1\} \subseteq \LR(\tau)$.
\end{lemma}

\begin{proof}
Set $I = \LR(\tau)$
and $\alpha = \st(\tau|_I)$.
By Proposition \ref{prop:union12}, we have $\alpha \in \fkI(12)$, meaning that $\alpha$ has the form $\alpha = \sigma \ominus \sigma^{-1}$ or $\alpha = \sigma \ominus 1 \ominus \sigma^{-1}$.
 Thus, every $(x,y) \in \Cyc(\alpha)$ satisfies $y \ge \lfloor k/2 \rfloor + 1$.
 Therefore, if $(i,i') \in \Cyc(\tau)$ with $i,i' \in I$, then $i' \ge \lfloor k/2 \rfloor + 1$.

Set $J = [n] \smallsetminus I$, and let $(j,j') \in \Cyc(\tau)$ with $j,j' \in J$. Since $(j,j') \in \Cyc(\tau)$ is not a left-to-right minimum, there is some $(i,i') \in \Cyc(\tau)$ to its left, and since $\tau \in \fkI_n(123)$ we have $i,i' \in I$.
Then $j > i \ge \lfloor k/2 \rfloor + 1$. Therefore every element of $J$ is greater than $\lfloor k/2 \rfloor + 1$.
\end{proof}

\begin{lemma} \label{lem:ltr-min2}
Let $[n] = I \sqcup J$ be a partition of $[n]$. Set $k = |I|$ and assume that $\{1, \ldots, \lfloor k/2 \rfloor + 1\} \subseteq I$. Let $\alpha \in \fkI_k(12)$ and $\beta \in \fkI_{n-k}(12)$. Define $\tau \in \fkI_n$ to be the unique involution with $\st(\tau|_I) = \alpha$ and $\st(\tau|_J) = \beta$.
Then $\tau \in \fkI(123)$ and $I = \LR(\tau)$.
\end{lemma}

\begin{proof}
That $\tau \in \fkI_n(123)$ follows immediately from Proposition \ref{prop:union12}; we need only to prove $I = \LR(\tau)$. Since $\alpha \in \fkI_k(12)$, it has the form $\alpha = \sigma \ominus \sigma^{-1}$ or $\alpha = \sigma \ominus 1 \ominus \sigma^{-1}$. 
Since $\{1, \ldots, \lfloor k/2 \rfloor + 1\} \subseteq I$, we can make these observations:
\begin{itemize}
\item[(a)] If $(i,i') \in \Cyc(\tau)$ with $i,i' \in I$, then $i \in \{1, \ldots, \lceil k/2 \rceil\}$;
\item[(b)] If $(x,y) \in \Cyc(\tau)$, then $y \notin \{1, \ldots, \lfloor k/2 \rfloor\}$;
\item[(c)] $(a,\lfloor k/2 \rfloor + 1) \in \Cyc(\tau)$ for some $a$.
\end{itemize}
(When $k$ is odd, note that $\lceil k/2 \rceil = \lfloor k/2 \rfloor + 1$ is a fixed point).

Observations (a) and (b) together imply that, if $(i,i') \in \Cyc(\tau)$ with $i,i' \in I$, then there is no $(x,y) \in \Cyc(\tau)$ with $y < i$ --- meaning $(i,i')$ is a left-to-right minimum of $\tau$, and $(i,i') \in \LR(\tau)$.
Therefore $I \subseteq \LR(\tau)$. On the other hand, if $(j,j') \in \Cyc(\tau)$ with $j,j' \in J$, then $j > \lfloor k/2 \rfloor + 1$.
Observation (c) above now implies that $(j,j')$ is to the right of some cycle $(a,\lfloor k/2 \rfloor + 1)$, and thus $(j,j')$ is not a left-to-right minimum and $(j,j') \notin \LR(\tau)$.
Therefore $I \supseteq \LR(\tau)$.
\end{proof}

\begin{proof}[Proof of Theorem~\ref{t:I_3-avoidance}~(c)]
Lemmas~\ref{lem:ltr-min1} and~\ref{lem:ltr-min2} establish a one-to-one correspondence between $\fkI_n(123)$ and the set of tuples $(k,I,\alpha,\beta)$ such that $\{1, \ldots, \lfloor k/2 \rfloor + 1\} \subseteq I \subseteq [n]$, $|I| = k$, $\alpha \in \fkI_k(12)$, and $\beta \in \fkI_{n-k}(12)$.
The condition $[ \lfloor k/2 \rfloor +1 ] \subseteq I$ is equivalent to $[n] \smallsetminus I$ being a subset of $\{\lfloor k/2 \rfloor + 2, \ldots, n\}$, so there are $\binom{n-\lfloor k/2 \rfloor - 1}{n-k}$ options for $I$.
Therefore, for a fixed $k$, the number of tuples with the properties above equals
\[ \binom{n-\lfloor k/2 \rfloor - 1}{n - k} \cdot |\fkI_k(12)| \cdot |\fkI_{n-k}(12)|. \]
Since $|\fkI_m(12)| = \left\lfloor \frac{m}{2} \right\rfloor !$, the result follows by summing over $k$ from $1$ to $n$.
\end{proof}

\subsection{Relating $\fkI$-avoidance to $\fkF$-avoidance}\

Recall the well known fact that the exponential generating function for $\fkF$ and $\fkI$ are
\[
F(x) = e^{x^2/2} \quad \mbox{and} \quad I(x) = e^{x^2/2} \cdot e^x.
\]
Our results in this section demonstrate for $\cR \subseteq \fkF$ an equivalent relationship between the exponential generating functions of $\fkI(\cR)$ and $\fkF(\cR)$.
Furthermore, we explain how term $e^x$ tracks fixed points in this setting.

For $\pi \in \fkS_n$, let $\Fix(\pi) = \{i \in [n] : \pi(i) = i\}$ be the set of fixed points and $\fix(\pi) = |\Fix(\pi)|$  be the number of fixed points.
For $\cR \subseteq \fkI$, $S \subseteq [n]$, and $m \in \mathbb{N}$, let
\[ \fkI_{n,S}(\cR) = \{\tau \in \fkI_n(\cR) : \Fix(\tau) = S\} \quad \text{and} \quad \fkI_{n,m}(\cR) = \{\tau \in \fkI_n(\cR) : \fix(\tau) = m\}. \]
That is, $\fkI_{n,S}(\cR)$ is the set of involutions in $\fkI_n(\cR)$ whose set of fixed points is $S$, and $\fkI_{n,m}(\cR)$ is the set of involutions in $\fkI_n(\cR)$ with $m$ fixed points.

\begin{proposition} \label{prop:fixedpoints}
Let $\cR \subseteq \fkF$, let $S \subseteq [n]$, and set $m = |S|$.
Then 
\[
|\fkI_{n,S}(\cR)| = |\fkF_{n-m}(\cR)|,\quad \mbox{hence} \quad|\fkI_{n,m}(\cR)| = \binom{n}{m} |\fkF_{n-m}(\cR)|.
\]
\end{proposition}

\begin{proof}
For $S \in \binom{[n]}{m}$, define $\Phi : \fkI_{n,S}(\cR) \to \fkF_{n-m}(\cR)$ by $\Phi(\tau) = \st(\tau|_{[n] \smallsetminus S})$ --- that is, $\Phi$ removes the fixed points from $\tau$.
Clearly $\Phi$ is injective, and $\Phi$ is surjective precisely because no involution in $\cR$ has a fixed point.
Indeed, if $\rho \in \fkF_{n-m}(\cR)$ and we insert fixed points into $\rho$ to obtain an involution $\tau$ with $\Fix(\tau) = S$, then any occurrence in $\tau$ of some $\sigma \in \cR$ would involve only the $2$-cycles of $\tau$ and hence would be an occurrence of $\sigma$ in $\rho$.
Thus, $\Phi$ is a bijection, proving the first equality of the proposition.
The second equality follows by summing over all $S \in \binom{[n]}{m}$.
\end{proof}

\begin{corollary}
\label{c:fpf-to-I}
For $\cR, \cR' \subseteq \fkF$, if $|\fkF_n(\cR)| = |\fkF_n(\cR')|$ for all $n$, then $|\fkI_{n,S}(\cR)| = |\fkI_{n,S}(\cR')|$ and $|\fkI_{n,m}(\cR)| = |\fkI_{n,m}(\cR')|$ for all $n$ and all $S \subseteq [n]$ and all $m$. \hfill $\qed$
\end{corollary}
Note Corollary~\ref{c:stanley} is a special case of Corollary~\ref{c:fpf-to-I}.

We now use Proposition \ref{prop:fixedpoints} to obtain a result on the bivariate exponential generating functions whose parameter is the number of fixed points. For $\cR \subseteq \fkF$, define
\[ F_\cR(x) = \sum_{n\ge0} |\fkF_n(\cR)|\, \frac{x^n}{n!} \quad \text{and} \quad I_{\cR}(x,t) = \sum_{n\ge0} \sum_{m=0}^n |\fkI_{n,m}(\cR)|\,t^m\,\frac{x^n}{n!}. \]
For example, $F_\varnothing(x) = e^{x^2/2}$ and $I_\varnothing(x,t) = e^{tx + x^2/2}$.

\begin{proposition}
\label{p:fixed-point}
If $\cR \subseteq \fkF$, then $I_\cR(x,t) = e^{tx} F_\cR(x)$.
\end{proposition}

\begin{proof}
We have
\[ [t^m x^n]\,e^{tx} F_\cR(x) = \frac{1}{m!} \cdot \frac{1}{(n-m)!} \cdot |\fkF_{n-m}| = \frac{1}{n!} \binom{n}{m} |\fkF_{n-m}| = [t^m x^n]\, I_\cR(x,t), \]
using Proposition \ref{prop:fixedpoints} for the last equality.
\end{proof}

Setting $t = 1$, we see $I_\cR(x,1) = e^x F_\cR(x)$, which is Equation~\eqref{eq:fixed-points}.

As an example, we enumerate $\fkI(2143)$: the generating function identity of Proposition~\ref{p:fixed-point} using $|\fkF_{2k}(2143)| = k!$ yields
\[ |\fkI_n(2143)| = \sum_{k=0}^{\lfloor n/2 \rfloor} \binom{n}{2k} k!, \]
which appears in OEIS as \cite[A084261]{sloane2018line}.

\subsection{Remarks on asymptotic enumeration}
Let $a_n$ be a sequence of complex numbers and let $c \ge 0$. We say that $a_n$ is \textbf{of exponential order $c^n$}, written $a_n \bowtie c^n$, if $\limsup_{n\to\infty} {|a_n|}^{1/n} = c$. This condition is equivalent to
\[ c = \inf \left\{ r>0 : \lim_{n\to\infty} \frac{|a_n|}{r^n} = 0 \right\}. \]
This is a coarse measure of asymptotic growth that can only distinguish between different bases of exponential growth.

By Corollary \ref{c:dichotomy}, for each $\tau \in \fkI$, either $|\fkI_n(\tau)| \bowtie c^n$ for some finite $c$ or $|\fkI_n(\tau)|$ is bounded below by $\lfloor n/2 \rfloor !$. The former case happens when $\tau \in \fkI(12)$, and here $\fkI$-avoidance coincides with ordinary pattern avoidance by Theorem \ref{t:singleton}. In the latter case, $\fkI$-avoidance is a distinct phenomenon. In this paper we have enumerated four $\fkI$ classes (up to geometric symmetry) that fall into this latter case: $\fkI(12)$, $\fkI(132)$, $\fkI(2143)$, and $\fkI(123)$.

\begin{table}
\[ \begin{array}{c|c}
\tau & |\fkI_n(\tau)| \\
\hline
12 & = (n/2)! \\
132 & \sim \frac12 e^{1/8} \, e^{\sqrt{n/2}} \, (n/2)! \\
2143 & \sim \frac12 e^{-1/2} \, e^{\sqrt{2n}} \, (n/2)! \\
123 & \sim \sqrt{\frac{2\pi}{27} n} \, {\left(\frac{2}{\sqrt{3}}\right)}^n \, (n/2)!
\end{array} \]
\caption{Asymptotic formula for $|\fkI_n(\tau)|$, for various involutions $\tau$, valid for even $n$. For $132$ and $2143$, the formula is taken from the existing OEIS entries, and we use Stirling's formula to convert the factor of $(n/2)^{n/2}$ to a factor of $(n/2)!$. For $123$, we obtain the formula directly from our enumeration in Theorem \ref{t:I_3-avoidance}~(c) by finding a Gaussian curve that the terms of the sum converge to as $n \to \infty$; the details are beyond the scope of this paper.\label{table:asymptotics}}
\end{table}

For these four classes and in general, it is natural to ask: if the enumeration is greater than $\lfloor n/2 \rfloor !$, then how much greater is it? Since $\frac{|\fkI_n|}{\lfloor n/2 \rfloor !} \bowtie {\left(\sqrt{2}\right)}^n$, we know that the exponential order of $\frac{|\fkI_n(\tau)|}{\lfloor n/2 \rfloor !}$ is between $1^n$ and ${\left(\sqrt{2}\right)}^n$. Table \ref{table:asymptotics} lists the asymptotic enumerations for the four classes, restricting to even $n$. For $\tau \in \{12, 132, 2143\}$, we find that $\frac{|\fkI_n(\tau)|}{\lfloor n/2 \rfloor !} \bowtie 1^n$; that is, $|\fkI_n(\tau)|$ is a sub-exponential function times $\lfloor n/2 \rfloor !$. However, $\frac{|\fkI_n(123)|}{\lfloor n/2 \rfloor !} \bowtie {\left(\frac{2}{\sqrt{3}}\right)}^n$. Why should the first three all be within a sub-exponential factor of each other while the fourth is higher?

If there is an underlying mathematical reason for this, it may become clearer if we look at the asymptotic enumeration of $\fkI(\tau)$ for other small involutions $\tau$ that $\fkI$-contain $12$. A reasonable next step would be to take $\tau$ of size $4$ that $\fkI$-contains $12$, in which case the remaining patterns to investigate (up to geometric symmetry) are $1234$, $1243$, $1324$, $1432$, and $4231$. In particular, we can hope that the enumeration of $\fkI(1324)$ and $\fkI(4231)$ is easier than the notoriously intractable $\fkS(1324)$ and $\fkS(4231)$.

\section{Future directions}
\label{sec:future}

\subsection{Connections to geometry}
One early application of pattern avoidance appears in the context of Schubert calculus, where 2143-avoiding or vexillary permutations take special prominence~\cite{lascoux1982structure}.
Abe and Billey have a wonderful survey on this topic~\cite{abe2016consequences}.
The results in Section~\ref{sec:mcgovern} are an extension of this theory.
One of the motivations of our work is to highlight the naturalness of $\fkI$-avoidance and $\fkF$-avoidance from a geometric perspective.

The \textbf{(complete) flag variety} $FL(n)$ is the space \textbf{flags}, sequences of vector spaces
\[
\{0\} = V_0 \subsetneq V_1 \subsetneq \dots \subsetneq V_n = \mathbb{C}^n.
\]
The left action by the group $B$ of invertible lower triangular matrices decomposes $FL(n)$ into cells called \textbf{Schubert varieties} indexed by $\fkS_n$.
Many geometric properties of Schubert varieties are characterized by pattern avoidance.

If one instead acts on $FL(n)$ by the orthogonal group $O(n)$ or the symplectic group $Sp(n)$ ($n$ even), one obtains varieties indexed by $\fkI_n$ and $\fkF_n$.
Such varieties are generally referred to as $K$-orbits.
McGovern's results are some of the first to characterize geometric properties of these varieties in terms of pattern avoidance.
Hamaker, Marberg and Pawlowski have shown many other properties of $K$-orbits are governed by $\fkI$ and $\fkF$-avoidance~\cite{hamaker2018involution,hamaker2019schur,hamaker2020fixed}.

Additionally, one can act on $FL(n)$ by $GL(p) \times GL(q)$ with $p + q = n$.
This action has orbits indexed by objects called \textbf{clans}, which are involutions whose fixed points are labeled $+$ or $-$.
The geometric properties of these orbits can also be described by pattern avoidance.
For example, Wyser has shown orbit closures indexed by $3412$-avoiding clans correspond to Richardson varieties~\cite{wyser2013schubert}.
Other work by Woo--Wyser~\cite{woo2015combinatorial} and Woo--Wyser--Yong~\cite{woo2018governing}, give clan pattern avoidance criteria for many other properties of these orbits.
Setting $t=2$ in Proposition~\ref{p:132} and Proposition~\ref{p:fixed-point} gives some enumerative results for clans.
We hope the permutation patterns community will make a deeper study of enumerative properties for pattern avoidance of clans.

\subsection{Other involution statistics}
The interplay between pattern avoidance and permutation statistics has grown into a rich area.
See for example~\cite{dokos2012permutation} and subsequent work.
Recently, Dahlberg began studying the interplay between ordinary pattern avoidance for involutions and permutation statistics~\cite{dahlberg2017permutation}.
We believe it would be interesting to investigate such properties for $\fkI$ and $\fkF$-avoidance.

As part of their work on the combinatorics of $K$-orbits, Hamaker, Marberg and Pawlowski have defined many novel statistics for involutions and fixed-point-free involutions.
Since these statistics are geometrically natural, they deserve further study.
Many of them are amenable to investigation from a pattern avoidance perspective.
We give a couple of examples and offer some questions to investigate.

For $\tau \in \fkI_n$, the \textbf{involution code} is the sequence $
	\hat{c}_1(\tau),\hat{c}_2(\tau), \dots, \hat{c}_{n-1}(\tau)$ with 
\[
\hat{c}_i(\tau) = |\{j\in [n]:\tau(j) \leq i < j \ \mbox{and}\ \tau(i) > \tau(j)\}|.
\]
Similarly, for $\rho \in \fkF$, the \textbf{fixed-point-free code} has
\[
\hat{c}^{FPF}_i(\rho) = |\{j\in [n]:\rho(j) < i < j \ \mbox{and}\ \rho(i) > \rho(j)\}|.
\]
Without the conditions $\tau(j) \leq i$ and $\rho(j) < i$, these would be the usual \textbf{code} of a permutation $\pi \in \fkS$ (i.e., the Lehmer code).
In the ordinary case, the pair $(i,j)$ forms an inversion.
The pairs $(i,j)$ counted by $\hat{c}_i(\tau)$ and $\hat{c}_i^{FPF}(\rho)$ are called \textbf{visible} and \textbf{FPF-visible inversions}, respectively.
For $\tau \in \fkI_n$ and $\rho \in \fkF_m$, the quantities $\sum_{i=1}^{n-1} \hat{c}_i(\tau)$ and $\sum_{i=1}^{2m-1} \hat{c}^{FPF}_i(\rho)$ measure the rank of these permutations in the restriction of the Bruhat order on $\fkS_n$ to $\fkI_n$ and $\fkF_{2n}$, respectively~\cite{hamaker2018involution}.
Several authors have studied pattern avoidance in codes, including~\cite{albert2003regular,corteel2016patterns,mansour2015pattern}.
We propose that pattern avoidance in involution and fixed-point-free codes is a topic that also deserves investigation.

By analogy with visible inversions,  for $\tau \in \fkI$ say $i$ is a \textbf{visible descent}  if $(i,i+1)$ is a visible inversion.
Similarly, for $\rho \in \fkF$, we say $i$ is an \textbf{FPF-visible descent} if $(i,i+1)$ is an FPF-visible inversion.
We propose that visible and FPF-visible descents are natural objects to consider.
While the Maj statistic associated to the sets of visible and FPF-visible descents are not equidistributed with the number of visible and FPF-visible inversion counts, we still believe interesting enumerative properties remain to be discovered.
Additionally, the interaction between (FPF-)visible descents and $\fkI$ and $\fkF$-avoidance provides a rich source of research questions.

\subsection{$\fkI/\fkF$-bases and enumeration}
One of the most important questions in pattern avoidance is identifying growth rates for classes.
We suspect Theorem~\ref{t:basis} offers a new perspective on growth rates for $\fkI_\fkS(\Pi)$ and $\fkF_\fkS(\Pi)$.
Recall for $\pi \in \{123,132,213,321\}$ we have $|\fkI_{\fkS_n}(\pi)| = \binom{n}{\lfloor n/2 \rfloor}$ and for $\pi \in \{231,312\}$ we have $|\fkI_{\fkS_n}(\pi)| = 2^n-1$ (see Table~\ref{tab:bases-algorithm}).
Meanwhile, the $\fkI$-basis of $\fkI_\fkS(\pi)$ for $\pi \in \{123,132,213,321\}$ consists of elements of $\fkI_3 \sqcup \fkI_5 \sqcup \fkI_6$ while for $\pi \in \{231,312\}$ the $\fkI$-basis consists of elements of $\fkI_4$.
It seems plausible that the length and quantity of elements in the $\fkI$-basis dictates asymptotic growth rates of $\fkI_\fkS$-classes (and likewise for $\fkF_\fkS$ classes).
A logical starting point would be to examine the $\fkI_\fkS$-classes for $\pi \in \fkS_4$.
The current state of the art appears in~\cite{bona2016pattern}, but the relative growth rates for such classes remains an open problem.
We hope the study of $\fkI$-bases will shed new light on this problem.

\subsection{Other cycle types} 
The key difference between $\fkI/\fkI'/\fkF$-containment and the usual notion of pattern containment is that we insist cycle structure be preserved.
One can easily identify notions of pattern avoidance that extend those we study to arbitrary cycle structure.
Recently, the combinatorics community has made a significant effort to understand the interaction between pattern containment/avoidance and the cycle structure of permutations~\cite{bona2019cyclic,gaetz2020stable}.
We propose that these questions should be revisited with the requirement that cycle structure be preserved.

%%%%%%%%%%%%%%%%%%%%%%%%%%%%%%%%%%%%%%%%%%%%%%%%%%%%%%%%%%%%
%
%  Bibliography
%
%%%%%%%%%%%%%%%%%%%%%%%%%%%%%%%%%%%%%%%%%%%%%%%%%%%%%%%%%%%%

\bibliographystyle{amsalpha} 
\newcommand{\etalchar}[1]{$^{#1}$}
\providecommand{\bysame}{\leavevmode\hbox to3em{\hrulefill}\thinspace}
\providecommand{\MR}{\relax\ifhmode\unskip\space\fi MR }
\providecommand{\MRhref}[2]{%
  \href{http://www.ams.org/mathscinet-getitem?mr=#1}{#2}
}
\providecommand{\href}[2]{#2}

\end{document}